\documentclass[a4paper, 12pt]{article}
\usepackage{hyperref}
\usepackage[utf8]{inputenc}
\usepackage{amsmath,amssymb,amsthm}
\usepackage[a4paper,hmargin=2.5cm,vmargin=2.5cm]{geometry}
\usepackage{enumitem}
\usepackage{cleveref}
\usepackage{bm}
\usepackage[numbers,sort]{natbib}
\usepackage{comment}
\usepackage{xcolor}

\newcommand{\AAA}{\mathbb{A}}

\newcommand{\ZZ}{\mathbb{Z}}

\newcommand{\CC}{\mathbb{C}}
\newcommand{\KK}{\mathbb{K}}

\newcommand{\bx}{\mathbf{x}}
\newcommand{\by}{\mathbf{y}}
\newcommand{\bz}{\mathbf{z}}
\newcommand{\balpha}{\bm{\alpha}}
\newcommand{\bomega}{\bm{\omega}}

\DeclareMathOperator{\Vect}{Vec}
\DeclareMathOperator{\Div}{div}
\DeclareMathOperator{\Aut}{Aut}
\DeclareMathOperator{\SAut}{SAut}

\newtheorem{proposition}{Proposition}
\newtheorem*{proposition*}{Proposition}
\newtheorem{remark}{Remark}
\newtheorem{lemma}{Lemma}
\newtheorem{corollary}{Corollary}

\newtheorem{theorem}{Theorem}
\newtheorem{problem}{Problem}

\theoremstyle{definition}
\newtheorem{notation}{Notation}
\newtheorem{definition}{Definition}
\newtheorem{example}{Example}

\begin{document}

\title{The Lie algebra of polynomial vector fields on the affine space with constant divergence is $1.5$-generated}

\author{Ivan Beldiev\thanks{The work of the first author is supported by the grant RSF25-11-00302 (Sections 2 and 3)} \and Gleb Pogudin\thanks{GP was supported by the French ANR-22-CE48-0008 OCCAM and ANR-22-CE48-0016 NODE projects (Sections 1 and 4)}}

\newcommand{\Addresses}{{% additional braces for segregating \footnotesize
  \bigskip
  \footnotesize

Ivan Beldiev, \textsc{HSE University, Faculty of Computer Science, Pokrovsky Boulevard 11, Moscow, 109028 Russia}\par\nopagebreak
\textit{E-mail address:} \texttt{ivbeldiev@gmail.com, isbeldiev@hse.ru}

  \medskip

Gleb Pogudin, \textsc{LIX, CNRS, École polytechnique, Institut Polytechnique de Paris, Paris, France}\par\nopagebreak
\textit{E-mail address:} \texttt{gleb.pogudin@polytechnique.edu}

}}

\date{}

\maketitle

\begin{abstract} 
A Lie algebra is said to be $1.5$-generated if every nonzero element can be completed to a two-element generating set.
We prove that the Lie algebra of polynomial vector fields with constant divergence on the affine space is $1.5$-generated. 
We also show that the Lie algebra of all polynomial vector fields on the complex affine space is generated by two completely integrable elements.\end{abstract}

\section{Introduction}\

One of the most common ways to describe an algebraic structure is by its generators, and the question of the minimal cardinality of a generating set is very natural in this context.
In the present paper, we study generating sets of Lie algebras over a field of zero characteristic.
Any noncommutative Lie algebra requires at least two generators.
In fact, two generators are sufficient for any semisimple finite dimensional Lie algebra as established by Kuranishi~\cite[Section 2]{Kur}; in other words, such algebras are \emph{2-generated}.
In~\cite{Bois}, similar results were proved in positive characteristic.

In the case of inifinite-dimensional Lie algebras, one important situation to consider is the Lie algebra $\Vect(X)$ of algebraic vector fields on an affine variety $X$.
This Lie algebra consists of all derivations of the algebra of regular functions on $X$. In this setting, a series of results has been obtain by Andrist in~\cite{Andrist2019, Andrist2023, Andrist2024}.
It was established in~\cite{Andrist2019} that $\Vect(\mathbb{C}^n)$\footnote{In fact, the proof works for any field of zero characteristic} is 3-generated.
Later~\cite{Andrist2023,Andrist2024}, explicit finite generating sets were exhibited for the Lie algebras of vector fields on the smooth quadric $SL(2,\CC)$, the quadratic cone $\{(x,y,z)\in\CC^3\mid xy = z^2\}$ and non-singular Danielewski surfaces.
An important feature of these constructions is that the generating vector fields are \emph{complete}, i.e. their flow maps exist for all complex times. 
This allows to find a finite number of one-parameter subgroups in the group of holomorphic automorphisms of these varieties that generate a subgroup acting infinitely transitively on them (for more detailed discussion, see~\cite[Section~4]{beldiev2025liealgebrapolynomialvector}). Note that this question had been addressed previously in \cite[Theorem 5.17]{AKZ} where it was proved that for any $n \geq 2$ there exist three $\mathbb G_a$-subgroups (which are the flow maps of locally nilpotent vector fields) in the group of algebraic automorphisms of $\mathbb A^n$ that generate a subgroup acting infinitely transitively on~$\mathbb A^n$.

The result of Andrist for $\Vect(\mathbb{C}^n)$ was sharpened in a recent paper by the first author~\cite{beldiev2025liealgebrapolynomialvector} who established that this Lie algebra is 2-generated.
However, the exhibited generators are not complete, so the question of generating $\Vect(\mathbb{C}^n)$ by only two complete vector fields remained open~\cite[Problem 4.3]{beldiev2025liealgebrapolynomialvector}. It is also of interest to answer the question whether classical subalgebras of $\Vect(\mathbb A^n)$ such as the subalgebras of vector fields with constant and zero divergence are $2$-generated. For $n = 2$, this question was answered positively in \cite[Propositions 1 and 11]{AndrArzh} and the explicit generators of these algebras found by the authors are complete.

Beyond 2-generation, there exists an even stronger property: a Lie algebra $\mathcal{L}$ is said to be \emph{$1.5$-generated} if, for every nonzero $x \in \mathcal{L}$, there exists $y \in \mathcal{L}$ such that $x$ and $y$ generate~$\mathcal{L}$.
Proving this property is more challenging as it is not sufficient to just exhibit explicit generators. 
Ionescu proved $1.5$-generatedness for finite dimensional complex algebras and extended this result to real semisimple algebras~\cite{Ion}. Further results in this direction include imposing extra conditions on the generators~\cite{Chist,Chist2}. We are not aware of any such result for Lie algebras of vector fields.

The main results of this paper are the following:
\begin{itemize}
    \item We prove that the Lie algebra of vector fields of constant divergence on $\mathbb{K}^n$ is $1.5$-generated, where $\mathbb{K}$ is a field of zero characteristic (Theorem~\ref{thm:constant} and Corollary~\ref{crl:constant}).
    \item We show that $\Vect(\mathbb{C}^n)$ can be generated by two complete vector fields for any $n \geq 2$ (Corollary~\ref{cor:complete}). 
    Moreover, we show that any complete vector field on $\mathbb{C}^n$ with nonconstant divergence can be extended to a set of two complete generators (Theorem~\ref{thm:complete}).
\end{itemize}

Our proofs rely on two key ingredients. 
First, we show that a generalized Euler derivation 
\[
\alpha_1 x_1 \partial_{x_1} + \ldots + \alpha_n x_n \partial_{x_n}
\]
can be used to ``split'' any other vector field into homogeneous components with respect to the natural $\mathbb{Z}^n$-grading.
Second, we prove that any nonzero vector field, after a suitable invertible change of coordinates, contains sufficiently many homogeneous components to generate the whole Lie algebra.

The rest of the paper is organized as follows.
Section~\ref{seq:grading} recalls the main notions and constructions used in the paper and contains a proof of the aforementioned ``splitting'' property of the generalized Euler derivation.
Section~\ref{seq:1.5} proves the $1.5$-generatedness of the Lie algebra of vector fields of constant divergence on $\mathbb{K}^n$.
We use some intermediate results from this proof to show, in Section~\ref{seq:complete}, that $\Vect(\mathbb{C}^n)$ can be generated by two complete vector fields.

\paragraph{Acknowledgements}
The first author is grateful to his academic supervisor Ivan Arzhantsev for posing the problem and useful discussions. The authors are grateful to Ivan Arzhantsev and Alexander Perepechko for comments on the manuscript.

%\section{Preliminaries and $\mathbb Z^n$-grading on $\Vect(\AAA^n)$}
\section{Preliminaries and auxiliary statements}
\label{seq:grading}\

Throughout the note, $\KK$ is the ground field of characteristic zero, and $\AAA^n$ stands for the $n$-dimensional affine space over $\KK$.

\begin{definition}[Generators of Lie algebra]
    Let $\mathcal{L}$ be a Lie algebra, and consider a subset $S \subseteq \mathcal{L}$.
    The minimal (with respect to inclusion) Lie subalgebra of $\mathcal{L}$ containing $S$ is called the Lie subalgebra \emph{generated by $S$} and is denoted by $\operatorname{Lie}(S)$.
    In other words, $\operatorname{Lie}(S)$ consists of all elements that can be obtained from the elements of $S$ by a finite composition of Lie brackets and linear combinations.

    We say that $\mathcal{L}$ is \emph{generated} by $S\subseteq \mathcal{L}$ if $\operatorname{Lie}(S) = \mathcal{L}$.
    For an integer $k$, we say that  a Lie algebra $\mathcal{L}$ is  \emph{$k$-generated} if it can be generated by a subset of cardinality $k$.
\end{definition}

Consider the Lie algebra $\Vect(\mathbb A^n)$ of algebraic (in other words, polynomial) vector fields on the affine space $\KK^n$ with coordinates $x_1, x_2, \ldots, x_n$. 
Denoting the partial derivative operator with respect to $x_i$ by $\partial_{x_i}$ for every $1 \leqslant i \leqslant n$, we can write this algebra explicitly as
\[
\Vect(\mathbb A^n) = \left\{ f_1 \partial_{x_1} + f_2 \partial_{x_2} + \ldots + f_n \partial_{x_n}\,\middle|\, f_1, f_2, \ldots, f_n\in\KK[x_1, x_2, \ldots, x_n] \right\},
\]
and the Lie bracket is given by
\[
[ f \partial_{x_i}, g \partial_{x_j}] = f \frac{\partial g}{\partial x_i} \partial_{x_j} - g\frac{\partial f}{\partial x_j} \partial_{x_i}
\]
for any $f,g\in\KK[x_1,x_2,\ldots, x_n]$ and for any $1\leqslant i,j\leqslant n$. 
%From now on, we will usually write $\partial_{x_i}$ instead of $\frac{\partial}{\partial_{x_i}}$.
The Lie algebra $\Vect(\AAA^n)$ is commonly denoted by $W_n$.

\begin{definition}[Divergence of a vector field]
The \emph{divergence} $\Div(\delta)$ of a vector field $\delta = f_1\partial_{x_1} + \ldots + f_n\partial_{x_n}$ is defined by
\[
\Div(\delta) = \frac{\partial f_1}{\partial x_1} + \ldots + \frac{\partial f_n}{\partial x_n}.
\]
\end{definition}

\begin{comment}
The Lie algebra $W_n = \Vect(\AAA^n)$ contains two classical subalgebras, namely the subalgebras of vector fields with constant and zero divergence. Recall that the divergence $\Div(\delta)$ of a vector field $\delta = f_1\partial_{x_1} + \ldots + f_n\partial_{x_n}$ is defined by
\[
\Div(\delta) = \frac{\partial f_1}{\partial x_1} + \ldots + \frac{\partial f_n}{\partial x_n}.
\]
\end{comment}

The Lie subalgebras of $\Vect(\AAA^n)$ with constant and zero divergence, respectively, are denoted by $\Vect^c(\AAA^n)$ and $\Vect^0(\AAA^n)$ and defined as follows:
\begin{align*}
\Vect^c(\AAA^n) &= \{\delta \in \Vect(\AAA^n) \mid \Div(\delta) \in \KK\},\\
\Vect^0(\AAA^n) &= \{\delta \in \Vect(\AAA^n) \mid \Div(\delta) = 0\}.
\end{align*}
The fact that these subspaces of $\Vect(\AAA^n)$ are Lie subalgebras follows from the formula 
\[
\Div([\delta_1, \delta_2]) = \delta_1(\Div(\delta_2)) - \delta_2(\Div(\delta_1)).
\]

The algebra $\Vect^0(\AAA^n)$ is a subspace of $\Vect^c(\AAA^n)$ of codimension $1$. Explicitly, $$\Vect^c(\AAA^n) = \langle x_1\partial_{x_1} + \ldots + x_n\partial_{x_n}\rangle \oplus \Vect^0(\AAA^n).$$

The Lie algebras $\Vect^c(\AAA^n)$ and $\Vect^0(\AAA^n)$ are invariant under regular automorphisms (see, for example, \cite[Proposition 4.5]{KR}) of $\AAA^n$, so their definitions do not depend on the chosen coordinates.

The algebras $\Vect^c(\AAA^n)$ and $\Vect^0(\AAA^n)$ are of special interest since they emerge naturally in various situations. 
For example, $\Vect^c(\AAA^n)$ is known to be the Lie algebra of the ind-group $\Aut(\AAA^n)$ of regular automorphisms of the affine space, while $\Vect^0(\AAA^n)$ is the Lie algebra of the ind-group $\SAut(\AAA^n)$ of regular automorphisms with Jacobian one (see, for example, \cite[Proposition 15.7.2]{FK}).

In the rest of this section, we discuss the standard $\mathbb Z^n$-grading on $W_n$ and show that for any $\delta \in W_n$ there exists $D \in W_n$ such that the Lie algebra generated by $\delta$ and $D$ contains all $\mathbb Z^n$-homogeneous components of $\delta$.

\begin{notation}[$\mathbb{Z}^n$-grading on $W_n$]
Recall that the standard $\mathbb Z^n$-grading on the polynomial ring $\KK[x_1, x_2, \ldots, x_n]$ induces a $\mathbb Z^n$-grading on $W_n$. 
This $\mathbb Z^n$-grading is defined on monomials by the following formula:
\[
\deg_{\mathbb Z^n}(x_1^{k_1}x_2^{k_2}\ldots x_n^{k_n}\partial_{x_i}) := (k_1, k_2, \ldots, k_i - 1, \ldots, k_n).
\]
\end{notation}
For $k = (k_1, k_2, \ldots, k_n)$, denote by $V_k$ the subspace of $W_n$ consisting of all elements of $\mathbb Z^n$-degree $k$. 
One can show
that $\dim V_k = n$ if all $k_1, k_2, \ldots, k_n$ are non-negative and $\dim V_k = 1$ if one of $k_1, k_2, \ldots, k_n$ is equal to $-1$ and the others are non-negative; otherwise, $\dim V_k = 0$.

Denote by $\mathfrak t$ the commutative $n$-dimensional Lie subalgebra of $W_n$ spanned by all derivations of the form $x_i\partial_{x_i}$, $i = 1, 2,\ldots, n$. 
The subspaces $V_k$ are precisely the weight subspaces of $\mathfrak t$ acting on $W_n$ by adjoint action, that is: 
\[
[x_j\partial_{x_j}, x_1^{k_1}x_2^{k_2}\ldots x_n^{k_n}\partial_{x_i}] = 
{\begin{cases}
 k_j x_1^{k_1}x_2^{k_2}\ldots x_n^{k_n}\partial_{x_i}, \text{ if } i \neq j,\\
 (k_i - 1) x_1^{k_1}x_2^{k_2}\ldots x_n^{k_n}\partial_{x_i}, \text{ if } i = j.
\end{cases}}
%(k_j - \delta_{ij})x_1^{k_1}x_2^{k_2}\ldots x_n^{k_n}\partial_{x_i},
\]
%where $\delta_{ij}$ is the Kronecker delta.

Now, we are going to study the weight decomposition of $W_n$ induced by just one element of $\mathfrak t$. Let us introduce the following notation.

\begin{notation}\label{not:eular}
    Let $\balpha = (\alpha_1, \ldots, \alpha_n) \in \KK^n$.
    For a monomial $x_1^{e_1}\ldots x_n^{e_n}\partial_{x_i}$, we define its $\balpha$-weight to be $\alpha_1 e_1 + \ldots + \alpha_n e_n - \alpha_i$.

    We also denote $D_{\balpha} := \alpha_1x_1 \partial_{x_1} + \ldots + \alpha_n x_n \partial_{x_n}$.
\end{notation}

%The following lemma is an easy statement from linear algebra.

\begin{lemma}\label{lem:eular_aux}
    Let $\balpha \in \KK^n$ and consider a polynomial vector field $\delta$ on $\AAA^n$.
    Let $w_1, \ldots, w_N$ be the possible $\balpha$-weights of the terms in $\delta$.
    We write $\delta = \delta_1 + \ldots + \delta_N$, where $\delta_i$ is the sum of the terms of weight $w_i$.
    Then the vector space $\langle \delta, [D_{\balpha}, \delta], [D_{\balpha}, [D_{\balpha}, \delta]], \ldots \rangle$ contains $\delta_1, \ldots, \delta_N$.
\end{lemma}

\begin{proof}
    Observe that, for every $r \geq 0$, we have
    \[
    \operatorname{ad}(D_{\balpha})^r (\delta) = \sum\limits_{i = 1}^N w_i^r \delta_i.
    \]
    Since all $w_i$'s are distinct, the right-hand sides of this equality for $r = 0, \ldots, N - 1$ form a nonsingular Vandermonde matrix multiplied by a vector $(\delta_1, \ldots, \delta_N)^T$.
    By inverting the matrix, we show that $\delta_1, \ldots, \delta_N$ belong to the vector space spanned by $\operatorname{ad}(D_{\balpha})^r$ for $0 \leqslant r < N$. 
    This finishes the proof.
\end{proof}

Now, we are going to show that the Lie algebra generated by an arbitrary vector field $\delta \in W_n$ and a suitably chosen $D\in \mathfrak t$ contains all $\mathbb Z^n$-homogeneous components of $\delta$.

\begin{lemma}\label{lem:eular}
    Let $\delta$ be a polynomial vector field on $\AAA^n$.
    Then there exists a nonempty Zariski open subset $U \subset \AAA^n$ such that, for every $\balpha := (\alpha_1, \ldots, \alpha_n) \in U$, the vector space
    \[
    \langle \delta, [D_{\balpha}, \delta], [D_{\balpha}, [D_{\balpha}, \delta]], \ldots \rangle
    \]
    contains all $\mathbb{Z}^n$-homogeneous components of $\delta$.
\end{lemma}

\begin{proof}
    Let $\bomega_1, \ldots, \bomega_N$ be the $\ZZ^n$-degrees appearing in $\delta$. 
    For any $\balpha \in \KK^n$, the $\balpha$-weights of the corresponding monomials are equal to the scalar products $(\balpha, \bomega_1), \ldots, (\balpha, \bomega_N)$.
    Therefore, due to~\Cref{lem:eular_aux}, it is sufficient to ensure that these scalar products are distinct.
    Every inequality $(\balpha, \bomega_i - \bomega_j) \neq 0$ defines an open subset which is nonempty as we can take $\alpha$ to be the $s$-th standard vector, where $s$ is the index of any nonzero component of $\bomega_i - \bomega_j$.
    The intersection of these open subsets gives the desired $U$.
\end{proof}

\section{$1.5$-generation of the Lie algebra $\Vect^c(\mathbb A^n)$}\label{seq:1.5}\

Throughout this section we will assume that $n \geq 2$.
The Lie algebra $\Vect^c(\AAA^1)$ is two-dimensional, so it is $1.5$-generated just as a vector space.

\begin{notation}
    Let $\delta = f_1 \partial_{x_1} + \ldots + f_n \partial_{x_n}$ be a polynomial vector field on $\AAA^n$, that is, $f_1, \ldots, f_n \in \KK[x_1, \ldots, x_n] =: \KK[\bx]$.
    \begin{itemize}
        \item Then $f_i$ will be referred to as the $\partial_{x_i}$-\emph{component} of $\delta$.
        \item For every monomial $m \in \KK[x_1, \ldots, x_n]$, the coefficient at $m$ in $f_i$ will be called the \emph{coefficient} of $\delta$ at $m\partial_{x_i}$.
        \item We define the degree of $\delta$ as $\deg_{\bx} \delta := \max\limits_{i=1, \ldots, n} \deg f_i - 1$.
    \end{itemize}
\end{notation}

Now, we are going to prove a series of lemmas.

\begin{lemma}\label{high_degree}
    Let $\delta = f_1\partial_{x_1} + \ldots + f_n\partial_{x_n}$ be a non-zero polynomial vector field on $\AAA^n$. 
    Then, for any positive integer $N$, there exists an invertible polynomial change of variables $y_1 = y_1(\bx), \ldots, y_n = y_n(\bx)$ such that $\deg_{\by} \delta > N$.
\end{lemma}

\begin{proof}
    We fix a positive integer $k$. 
    By reordering variables, if necessary, we will assume that $f_{n - 1}\neq 0$.
    We consider the following change of variables: 
    \[
    y_i = x_i \text{ for } 1 \leqslant i < n \quad \text{ and }\quad y_n = x_n - x_{n - 1}^k.
    \]
    In the new variables, we have
    \[
    \partial_{x_i} = \partial_{y_i}\text{ for } i \neq n - 1\quad \text{ and } \quad \partial_{x_{n-1}} = \partial_{y_{n-1}} - ky_{n-1}^{k-1}\partial_{y_n}.
    \]
    Therefore, $\delta$ can be written in the new coordinates as follows:
    \begin{equation}\label{change_1}
    \delta = f_1\partial_{y_1} + \ldots + f_{n-2}\partial_{y_{n-2}} + f_{n-1}(y_1,\ldots, y_{n-1}, y_n + y_{n-1}^k)(\partial_{y_{n-1}} - ky_{n-1}^{k-1}\partial_{y_n}) + f_n\partial_{y_n}.
    \end{equation}
    The $\partial_{y_{n - 1}}$-component of $\delta$ is, therefore, equal to $f_{n-1}(y_1, \ldots y_{n-1}, y_n + y_{n-1}^k)$. 
    If $f_{n-1}(x_1, \ldots, x_n)$ depends nontrivially on $x_n$, the degree of this polynomial can be made arbitrarily large by choosing sufficiently large $k$.
    
    If $f_{n - 1}$ does not involve $x_n$, we have $f_{n-1}(y_1, \ldots y_{n-1}, y_n + y_{n-1}^k) = f_{n - 1}(y_1, \ldots, y_n)$.
    We consider the $\partial_{y_n}$-component of $\delta$:
    \begin{equation}\label{eq:yn_component}
    -k f_{n - 1}(y_1, \ldots, y_{n}) y_{n - 1}^{k - 1} + f_n(y_1, \ldots y_{n-1}, y_n + y_{n-1}^k).
    \end{equation}
    Let $m$ be any monomial $m$ in $f_{n - 1}(y_1, \ldots, y_{n})$ and we denote its coefficient by~$c$. 
    There are only finitely many values of $k$ such that $ck$ can be equal to the sum of a subset of coefficients of $f_n(y_1, \ldots y_{n-1}, y_n + z)$.
    Therefore, for any larger $k$, \eqref{eq:yn_component} contains a monomial $my_{n - 1}^{k - 1}$ with nonzero coefficient.
    By taking $k$ large enough, we can make the degree of this monomial arbitrarily high.
\end{proof}

By~\Cref{high_degree} and the pigeonhole principle, we can assume that the degree of at least one of $f_i$ with respect to at least one of the variables $x_j$ is arbitrarily high, i.e., there is a monomial in $f_i$ divisible by an arbitrarily high power of $x_j$. 
In the next lemma, we show that this monomial can be assumed to be divisible by an arbitrarily high power of all variables $x_1, x_2, \ldots, x_n$ simultaneously. 

\begin{lemma}\label{high_degree_all}
    Let $\delta = f_1\partial_{x_1} + \ldots + f_n\partial_{x_n}$ be a non-zero polynomial vector field on $\AAA^n$ and $k$ be a positive integer. 
    Assume that there exist $1 \leqslant i, j \leqslant n$ such that $\deg_{x_j} f_i \geq nk$.
    Then there exists an invertible linear change of variables $y_1 = y_1(\bx), y_2 = y_2(\bx), \ldots, y_n = y_n(\bx)$ such that the expression of $\delta$ in $y$-coordinates contains a monomial divisible by $(y_1 y_2 \ldots y_n)^k$.
\end{lemma}

\begin{proof}
    Without loss of generality, assume that $j = 1$, i.e., $\deg_{x_1} f_i \geq nk$. 
    The desired property of the $y$-coordinates is equivalent to $\partial_{y_1}^k\partial_{y_2}^k\ldots \partial_{y_n}^k(\delta)$ being a nonzero vector field.
    A linear change of variables is equivalent to choosing a basis in the linear span $V = \langle \partial_{x_1}$, $\partial_{x_2}$, \ldots, $\partial_{x_n}\rangle$ of all partial derivatives. 
    Consider the derivations $\partial_{y_1} = \partial_{x_1}$ and $\partial_{y_s} = \partial_{x_1} + t\partial_{x_s}$ for $s = 2, 3, \ldots, n$ and some $t\in\mathbb K$ which form a basis of $V$ for any $t \neq 0$. 
    Note that $\partial_{y_1}^k\partial_{y_2}^k\ldots \partial_{y_n}^k(\delta) \neq 0$ is an open condition for $t$ in the Zariski topology on $\mathbb K$. 
    This condition is fulfilled for $t = 0$ since $\delta$ contains a monomial divisible by $x_1^{nk}$. Therefore, the condition $\partial_{y_1}^k\partial_{y_2}^k\ldots \partial_{y_n}^k(\delta) \ne 0$ defines an open nonempty subset of $\mathbb K$ and hence is satisfied for some $t \neq 0$. 
    This concludes the proof of the lemma. 
\end{proof}

\begin{lemma}\label{high_degree_all_2}
    Let $\delta = f_1\partial_{x_1} + \ldots + f_n\partial_{x_n}$ be a non-zero polynomial vector field on $\AAA^n$ and $k$ be a positive integer.
    Assume that one of the components of $\delta$ contains a monomial divisible by $(x_1x_2 \ldots x_n)^k$
    Then there exists an invertible linear change of variables $y_1 = y_1(\bx), y_2 = y_2(\bx), \ldots, y_n = y_n(\bx)$ such that, for every $1 \leqslant i \leqslant n$, the $\partial_{y_i}$-component of $\delta$ contains a monomial divisible by $(y_1 y_2 \ldots y_n)^k$.
\end{lemma}

\begin{proof}
    Without loss of generality, we can assume that $f_1$ contains a monomial $ax_1^{k_1}x_2^{k_2}\ldots x_n^{k_n}$ such that $k_1$, $k_2$, \ldots, $k_n \geq k$ (and $a \neq 0$). 
    Let $i \neq 1$ and consider the following change of variables: 
    \[
    y_1 = x_i,\quad y_i = x_1,\quad \text{and }\quad y_j = x_j  \text{for all } j\not\in \{1, i\}. 
    \]
    After this change of variables, the term $ax_1^{k_1}x_2^{k_2}\ldots x_n^{k_n}\partial_{x_1}$ of $\delta$ transforms into 
    \[
    ay_1^{k_i}y_2^{k_2}\ldots y_i^{k_1}\ldots y_n^{k_n}\partial_{y_i},
    \]
    so the $\partial_{y_i}$-component of $\delta$ contains a monomial divisible by $(y_1\ldots y_n)^{k}$. 
    The condition that the $\partial_{y_i}$-component contains a monomial divisible by $(y_1 \ldots y_n)^k$ is an open condition on the linear change of variables, that is, it defines a Zariski open subset of the algebraic group $GL(n,\mathbb K)$ consisting of all such linear changes.
    We have just shown that this set is also nonempty, so these conditions for  $i = 1, 2, \ldots, n$ are fulfilled simultaneously on an open non-empty subset of $GL(n, \mathbb K)$.
    Taking a point in this subset finishes the proof.
\end{proof}

For any monomial $m(\bx) = x_1^{k_1}x_2^{k_2}\ldots x_n^{k_n}$, denote $\partial_m = \partial_{x_1}^{k_1}\partial_{x_2}^{k_2}\ldots \partial_{x_n}^{k_n},$ which is a differential operator on $\KK[x_1,x_2, \ldots, x_n]$.

\begin{lemma}\label{many_non_zero}
    Let $\delta$ be a polynomial vector field on $\AAA^n$ and $(m_1, i_i), (m_2, i_2), \ldots, (m_r, i_r)$ a collection of pairs, where $m_j$ is a monomial in $\KK[\bx]$ and $1 \leqslant i_j \leqslant n$.
    Assume that, for every $1 \leqslant j \leqslant r$, the $\partial_{x_{i_j}}$-component of $\partial_{m_j}(\delta)$ is nonzero.
    There there exist $a_1, \ldots, a_n \in \KK$ such that, under a translation $y_1 = x_1 - a_1, \ldots, y_n = x_n - a_n$, the coefficient of $m_j(\by)\partial_{y_{i_j}}$ in $\delta$ is nonzero for every $1 \leqslant j \leqslant r$.
\end{lemma}

\begin{proof}
    Since the $\partial_{x_{i_j}}$-component of $\partial_{m_j}(\delta)$ is a non-zero polynomial, its value is non-zero on a non-empty open subset of $\mathbb A^n$. 
    Let $(a_1, \ldots, a_n)$ be a point in the intersection of these open subsets over all $1 \leqslant j \leqslant r$.
    Consider the translation $y_j = x_j - a_j$, $j = 1,2, \ldots, n$. 
    In these new coordinates, the point $(a_1, \ldots, a_n)$ becomes the origin, so the constant coefficient of the $\partial_{y_{i_j}}$-component of $\partial_{m_j}(\delta)$ is nonzero for every $j$.
    This is equivalent to the coefficient of $m_j(\by)\partial_{y_{i_j}}$ in $\delta$ being nonzero.
\end{proof}

\begin{notation}
    In a list or a product $\hat{a}$ means that $a$ is \emph{absent}.
    For example, $x_1 \ldots \hat{x}_i \ldots x_n$ is the product of all $x_j$'s without $x_i$.
\end{notation}

\begin{proposition}\label{change_final}
    Let $\delta = f_1\partial_{x_1} + \ldots + f_n\partial_{x_n}$ be a non-zero polynomial vector field on $\AAA^n$.
    Then, for any positive integer $k$, there exists a change of variables $y_1 = y_1(\bx), \ldots, y_n = y_n(\bx)$ such that, for every $1 \leqslant i \leqslant n$, the coefficients at $\partial_{y_i}$ and at $(y_1\ldots \widehat{y_i}\ldots y_n)^k\partial_{y_i}$ in $\delta$ are non-zero.
\end{proposition}

\begin{proof}
    Lemmas \ref{high_degree}, \ref{high_degree_all} and \ref{high_degree_all_2} impliy that, after performing a change of coordinates, we can assume that the $\partial_{x_i}$-component of $(\partial_{x_1}\partial_{x_2}\ldots\partial_{x_n})^k(\delta)$ is nonzero for every $1 \leqslant i \leqslant n$.
    In particular, all components of $\delta$ itself are not zero.
    We denote $m_i = (x_1\ldots \widehat{x_i}\ldots x_n)^k$ for any $1 \leqslant i \leqslant n$ and observe that all the components are nonzero for $\partial_{m_i}(\delta)$ as well.
    We apply~\Cref{many_non_zero} to the pairs $(1, 1), \ldots, (1, n), (m_1, 1), \ldots, (m_n, n)$.
    This gives a translation $y_1 = x_1 - a_1, \ldots, y_n = x_n - a_n$ such that, for every $1 \leqslant i \leqslant n$, the coefficients at $\partial_{y_i}$ and at $(y_1\ldots \widehat{y_i}\ldots y_n)^k\partial_{y_i}$ in $\delta$ are nonzero as desired.
\end{proof}

\begin{theorem}\label{thm:constant}
    For any $n \geq 2$ and for any non-zero $\delta \in \Vect(\mathbb A^n)$ there exists $D \in \Vect^c(\mathbb A^n)$ such that the Lie algebra generated by $\delta$ and $D$ contains the whole algebra $\Vect^c(\mathbb A^n)$.
\end{theorem}

\begin{proof}
    Let $\delta$ be a non-zero vector field on $\mathbb A^n$. 
    By~\Cref{change_final}, we can assume that $\delta$ has a non-zero coefficient in front of each $\partial_{x_i}$ and of each $(x_1\ldots \widehat{x_i}\ldots x_n)^5\partial_{x_i}$. 
    We apply~\Cref{lem:eular} to $\delta$ and let $\alpha_1, \ldots, \alpha_n$ be a point in the resulting open set such that $\alpha_1 + \ldots + \alpha_n \neq 0$ (which is also an open condition and can be fulfilled).
    We take $D = \alpha_1 x_1 \partial_{x_1} + \ldots + \alpha_n x_n \partial_{x_n}$.
    Then the Lie algebra generated by $\delta$ and $D$ contains all $\ZZ^n$-homogeneous components of $\delta$. 
    In particular, it contains, for every $1 \leqslant i \leqslant n$, $\partial_{x_i}$ and $(x_1 \ldots \widehat{x_i}\ldots x_n)^5\partial_{x_i}$ since these vector fields belong to one-dimensional $\mathbb Z^n$-homogeneous components of $W_n$. 
    Applying all possible $\partial_{x_j}$ to all possible $(x_1 \ldots \widehat{x_i}\ldots x_n)^5\partial_{x_i}$, we obtain all vector fields of the form $x_1^{k_1}\ldots,\widehat{x_i^{k_i}}\ldots x_n^{k_n}\partial_{x_i}$ with $k_1, \ldots, \widehat{k_i}, \ldots, k_n \leqslant 5$. 
    It follows from~\cite[Theorem 14]{Andrist2019} that these vector fields generate the Lie algebra $\Vect^0(\AAA^n)$ of vector fields with zero divergence. 
    Together with the vector field $D = \alpha_1 x_1 \partial_{x_1} + \ldots + \alpha_n x_n \partial_{x_n}$ with $\alpha_1 + \ldots + \alpha_n \neq 0$, we obtain the whole Lie algebra $\Vect^c(\AAA^n)$ of vector fields with constant divergence.
\end{proof}

\begin{corollary}\label{crl:constant}
    The Lie algebra $\Vect^c(\mathbb A^n)$ is $1.5$-generated.
\end{corollary}

\begin{proof}
    Let $\delta$ be a non-zero vector field on $\mathbb A^n$ with constant divergence. According to Theorem \ref{thm:constant}, there exists $D \in \Vect^c(\mathbb A^n)$ such that the Lie algebra generated by $\delta$ and $D$ contains $\Vect^c(\mathbb A^n)$. On the other hand, both $\delta$ and $D$ belong to $\Vect^c(\mathbb A^n)$, so the Lie algebra generated by these two vector fields coincides with $\Vect^c(\mathbb A^n)$.
\end{proof}

However, the following question remains open.

\begin{problem}
    %Is the Lie algebra $\Vect(\AAA^n)$ or $\Vect^0(\AAA^n)$ $1.5$-generated?
    Are the Lie algebras $\Vect(\AAA^n)$ and $\Vect^0(\AAA^n)$ $1.5$-generated?
\end{problem}

%%%%%%%%%%%%%%%%%%%%%%%%%%%%%%%%%%%%%%%%%

\section{Generating $W_n$ by two complete vector fields}\label{seq:complete}

In this section, suppose that $\KK$ is the field of complex numbers $\CC$. 
In this case, for any smooth affine algebraic variety $X$ and for any holomorphic (in particular, algebraic) vector field~$\delta$ on $X$, there is a classical notion of the flow map $\varphi_{\delta}$ of $\delta$ which we recall briefly. 
For $t\in \CC$ and $p\in X$, the value $\varphi_{\delta}(t,p) = \varphi_{\delta, t}(p)$ is defined as $\gamma(t)$, where $\gamma$ is the integral curve of $\delta$ (i.e., the map $\gamma\colon \Omega \to X$, where $\Omega\subseteq \CC$ is a neighbourhood of $0$ containing $t$, such that the tangent vector $\dot{\gamma}(s)$ is equal to the value of $\delta$ at $\gamma(s)$ for all $s\in \Omega$) with $\gamma(0) = p$. 
It follows from the existence and uniqueness theorem for ordinary differential equations that for any $p\in X$ the flow map $\varphi_{\delta,t}$ is defined for all sufficiently small values of~$t$. 
Moreover, the flow map is holomorphic and depends holomorphically on $t$.

%There is the following notion of a complete vector field.

\begin{definition}[Complete vector field]
    An algebraic vector field on a complex affine algebraic variety $X$ is called \emph{complete} (or \emph{completely integrable}) if its flow map is well-defined 
    for all complex times. 
\end{definition}

For a complete vector field $\delta$ on $X$, the map $\varphi_{\delta,t}$ is a holomorphic automorphism of~$X$ for any $t\in\CC$. 
Thus, $\varphi_{\delta}$ defines a one-parameter subgroup in the group of holomorphic automorphisms of $X$. 
Note that the flow map of a complete vector field $\delta$ is not necessarily algebraic even if $\delta$ itself is algebraic.

\begin{example}
Consider
$\delta_1 = z_1\frac{\partial}{\partial z_1}$ and $\delta_2 = z_1z_2\frac{\partial}{\partial z_1}$ on $\CC$ and $\CC^2$.
%are given, respectively, by the formulas
Their flow maps are given by
\begin{align*}
\varphi_{\delta_1,t}(z_1) &= \operatorname{exp}(t)z_1,\\
\varphi_{\delta_2,t}(z_1,z_2) &= (\operatorname{exp}(tz_2)z_1, z_2).
\end{align*}
The first flow map does not depend algebraically on $t$ (although it defines an algebraic automorphism of $\CC$ for any fixed $t\in\CC$), while the second one is not an algebraic automorphism of $\CC^2$ even for a fixed $t\in \CC$.
\end{example}

Now we are going to show that the Lie algebra $W_n$ of all polynomial vector fields on $\mathbb A^n$ can be generated by two complete vector fields (and hence there exist two one-parameter subgroups in the group of holomorphic automorphisms of $\mathbb A^n$ generating a group acting infinitely transitively on $\mathbb A^n$; see \cite[Corollary 8]{Andrist2019}). 
%Let us start with the following proposition.

\begin{proposition}\label{prop:non-constant}
    Let $\delta$ be a vector field on $\AAA^n$ with non-constant divergence, where $n \geq 2$.
    Then the Lie algebra generated by $\delta$ and $\Vect^0(\mathbb A^n)$ coincides with $W_n$.
\end{proposition}

\begin{proof}
Recall that, for $\delta_1 \in \Vect^0(\mathbb A^n)$ and $\delta_2 \in W_n$, we have 
    \[
    \Div([\delta_1, \delta_2]) = \delta_1 (\Div(\delta_2)).
    \]
    It is sufficient to show that this algebra contains a derivation with divergence $m$ for every monomial $m \in \KK[\bz]$.
    By acting by $\partial_{z_i} \in \Vect^0(\mathbb A^n)$ on $\delta$, we show this for $m = 1$.
    Since $\Div(\delta)$ is not a constant, one can obtain a derivation $\delta_1$ with linear divergence by acting on it by partial derivative operators.
    By performing an affine change of coordinates if necessary, we will further assume that $\Div(\delta_1) = z_1$. 
    By acting by $z_j \partial_{z_1} {\in \Vect^0(\mathbb A^n)}$ for $j = 2, \ldots, n$ on $\delta_1$ we obtain $\delta_2, \ldots, \delta_n$ with $\Div(\delta_i) = z_i$.

    Let $m = z_1^{d_1} \ldots z_n^{d_n}$ be any monomial. 
    We have $\frac{mz_1}{d_1 + 1}\partial_{z_1} - \frac{mz_2}{d_2 + 1}\partial_{z_2} \in \Vect^0(\mathbb A^n)$, so
    \[
    \Div\left( \left[\frac{mz_1}{d_1 + 1}\partial_{z_1} - \frac{mz_2}{d_2 + 1}\partial_{z_2}, (d_1 + 1)\delta_1\right] \right) = mz_1.
    \]
    Therefore, for any monomial divisible by $z_1$, the Lie algebra generated by $\delta$ and $\Vect^0(\mathbb A^n)$ contains a derivation whose divergence is equal to this monomial. 
    Similarly, this is true for a monomial divisible by any other variable, i.e., for any non-constant monomial. This concludes the proof.
\end{proof}

\begin{remark}
    Proof of Proposition~\ref{prop:non-constant} is valid in fact for any ground field of zero characteristic, not necessarily for $K = \mathbb{C}$.
\end{remark}

Now, we are ready to prove the main result of this section.

\begin{theorem}\label{thm:complete}
    For any $n \geq 2$ and any non-zero $\delta \in W_n$ with non-constant divergence, there exists a complete vector field $D \in W_n$ such that $W_n$ is generated by $\delta$ and $D$.
\end{theorem}

\begin{proof}
    Let $\delta$ be any non-zero vector field on $\AAA^n$ with non-constant divergence. By Theorem \ref{thm:constant}, there exists a vector field $D$ with constant divergence such that the Lie algebra generated by $\delta$ and $D$ contains $\Vect^c(\mathbb A^n)$. Moreover, according to the proof of this theorem, this vector field is complete. Indeed, it is equal to $D_{\balpha}$ in some coordinates, so its flow map is given by $$\varphi_{D_{\balpha}}(t)\colon (z_1, z_2, \ldots, z_n) \mapsto (z_1\exp(t\alpha_1), z_2\exp(t\alpha_2), \ldots, z_n\exp(t\alpha_n))$$
    and is clearly defined for any values $t_i\in\mathbb C$. By Proposition \ref{prop:non-constant}, the Lie algebra generated by $\delta$ and $\Vect^c(\mathbb A^n)$ is the whole $W_n$, so we are done.
\end{proof}

\begin{corollary}\label{cor:complete}
    The Lie algebra $W_n$ can be generated by two complete vector fields for any~$n \geq 2$.
\end{corollary}

\begin{proof}
    The statement follows easily from Theorem \ref{thm:complete}. Indeed, it suffices to take an arbitrary complete vector field $\delta$ with non-constant divergence. For example, we can take $\delta = z_1z_2\partial_{z_1}$ since the flow map $\varphi_\delta$ of $\delta$ is given by $\varphi_\delta \colon (z_1, z_2, \ldots, z_n) \mapsto (z_1\exp{tz_2}, z_2, \ldots, z_n)$ and is defined for all $t\in\mathbb C$. 
\end{proof}

\bibliographystyle{abbrvnat}
\bibliography{bib}

\Addresses

\end{document}